\documentclass{article}
\usepackage[autostyle]{csquotes}
\usepackage{extarrows,mathtools,amsmath,amssymb,enumerate,hyperref}

\usepackage{amsthm}
\newtheorem{theorem}{Theorem}[section]
\newtheorem{proposition}[theorem]{Proposition}
\newtheorem{lemma}[theorem]{Lemma}
\newtheorem{corollary}[theorem]{Corollary}
\theoremstyle{remark}
\newtheorem*{remark}{Remark}
\newtheorem{question}{Question}
\DeclareMathOperator{\Syl}{Syl}
\DeclareMathOperator{\SL}{SL}
\DeclareMathOperator{\im}{Im}

\renewenvironment{abstract}
 {\par\noindent\textbf{\abstractname.}\ \ignorespaces}
 {\par\medskip}
 
\newenvironment{acknowledgements}{%
  \renewcommand{\abstractname}{Acknowledgements}
  \begin{abstract}
}{%
  \end{abstract}
}

\newenvironment{funding}{%
  \renewcommand{\abstractname}{Funding}
  \begin{abstract}
}{%
  \end{abstract}
}

\newenvironment{conflicts}{%
  \renewcommand{\abstractname}{Conflicts of interests}
  \begin{abstract}
}{%
  \end{abstract}
}

\begin{document}
\title{Finite groups with integer harmonic mean of element orders}
\author{Iulia C\u at\u alina Ple\c sca\footnote{
Faculty of Mathematics of "Al. I. Cuza" University of Ia\c si, Romania, 
e-mail: dankemath@yahoo.com
ORCID: 0000-0001-7140-844X}\text{ } and Marius T\u arn\u auceanu\footnote{Faculty of Mathematics of
"Al. I. Cuza" University of Ia\c si, Romania,
e-mail: tarnauc@uaic.ro}}
\date{September 9, 2023}
\maketitle

\begin{abstract}
In this paper, we introduce a new function computing the harmonic mean of element orders of a finite group. We present a series of properties for this function, and then we study groups for which the value of the function is an integer.
\end{abstract}
\section{Introduction}
Throughout the article, let $G$ be a finite group and denote by $o(a)$ the order of an element $a\in G$. 

In the last three decades, an increasing number of functions and concepts from number theory have been adapted to group theory. For example, Leinster introduced what would later be called Leinster groups \cite{Leinster}, a group analogue of perfect numbers. Subsequently, a plethora of arithmetic functions that involve the orders of elements/subgroups of groups have been introduced. We refer the reader to \cite{article4} for a recent survey on this topic.

Special attention has also been given to various means of element orders: arithmetic in \cite{Silviu,article5} and geometric in \cite{article5}. In our paper, we will focus on the harmonic mean. Let us recall the functions that represent the starting point for our work.

Two of the most common functions in number theory associated to a positive integer $n\in\mathbb{N}^*$ are the number of its divisors $\tau(n)$ and the sum of its divisors $\sigma(n)$. These two functions have been adapted to group theory as follows. Given a finite group, $G$, $\tau(G)$ represents the number of its normal subgroups and $\sigma(G)$ the sum of their orders.

These latter two appear in the definition of a new function \[H(G)=|G|\frac{\tau(G)}{\sigma(G)}\,,\] introduced in \cite{article1} that was meant to generalize the arithmetic function $H(n)=n\tau(n)/\sigma(n)$.

The last function that inspired us, the sum of inverses of the orders of elements \[m(G)=\sum_{a\in G}\frac{1}{o(a)}\,,\] was introduced in \cite{article2} and, similarly to other sum functions, can give characterizations for commutativity, ciclicity and nilpotency.

Inspired by these, we introduce the function \enquote{the harmonic mean of element orders of a finite group}, i.e.
\begin{equation}
h_m(G)=\frac{|G|}{m(G)}\,.
\end{equation}

Throughout this article we put a dent in the following problem:

\begin{question}
Which are the finite groups $G$ with $h_m(G)\in\mathbb{N}$?
\end{question}

We start by giving some immediate properties of the newly introduced function including a lower bound that is reached only by finite $p$-groups. We continue by characterizing the finite $p$-groups with integer $h_m$. In the end, we observe that $h_m^{-1}(2)=\{C_4, D_8\}$ and study $h_m^{-1}(3)$.

We use the following notations for some of the most common groups. For $m,n\in\mathbb{N}^*$, $C_n$ denotes the cyclic group of order $n$, $D_n$ the dihedral group of order $n$, $S_n$ the symmetric group of order $n$, $SD_{n}$ the semidihedral group of order $n$, $SL(m,n)$ the special linear group of degree $m$ over a field of $n$ elements. In addition, $Q_8$ is the quaternion group. The rest of the notations are standard. Basic notions and results on groups can be found in \cite{book2}.

\section{Main results}

A series of inequalities for the function $m$ are given in \cite{article2} (Lemmas 1.3, 2.2, 2.3, 2.4, 2.6). We summarize these in the first proposition:

\begin{proposition}[\cite{article2}]\label{propm}
The following properties hold for the function $m$:
\begin{enumerate}[a)]
\item
If $|G|=n$, then $m(C_n)\leq m(G)$, with equality if and only if $G\cong C_n$.
\item
If $H\leq G$, then $m(H)\leq m(G)$, with equality if and only if $H=G$.
\item
If $N\unlhd G$, then $m(G/N)\leq m(G)$, with equality if and only if $N=1$.
\item
If $P$ is a normal cyclic Sylow $p$-subgroup of $G$: $G \in \Syl_p(G)$, then $m(Px) \geq m(P)/o(Px)$, where $Px \in G/P$. Equality holds if and only if x centralizes $P$. Also, $m(G) \geq m(P)m(G/P)$ with equality if and only if $P$ is central in $G$.
\item
If $G_1, G_2$ are finite groups, then $m(G_1 \times G_2) \geq m(G_1)m(G_2)$, with equality if and only if the orders of the groups are coprime: $\gcd(|G_1|, |G_2|) = 1$.

\end{enumerate}
\end{proposition}

We deduce a corresponding proposition for $h_m$:

\begin{proposition}
The following properties hold:
\begin{enumerate}[a)]
\item
If $|G|=n$, then $h_m(G)\leq h_m(C_n)$. Equality holds if and only if $G\cong C_n$.
\item
For a subgroup $H\leq G$, it follows that $h_m(G)\leq [G:H]h_m(H)$. Equality holds if and only if $H=G$.
\item
For a normal subgroup $H\lhd G$, it follows that $h_m(G)\leq |H|h_m(\frac{G}{H})$. Equality holds if and only if $H$ is the trivial subgroup.
\item
If $P$ is a normal cyclic Sylow $p$-subgroup of $G$: $G \in \Syl_p(G)$, then $h_m(G)\leq h_m(P)h_m(\frac{G}{P})$. Equality holds if and only if $P$ is central in $G$.
\item
$h_m$ is multiplicative: for all finite groups $G_1, G_2$ of coprime orders, we have $h_m(G_1\times G_2)=h_m(G_1)h_m(G_2)$. This shows that the study of the function $h_m$ for finite nilpotent groups can be reduced to $p$-groups.
\end{enumerate}
\end{proposition} 

Throughout this study, a lower bound for the function $h_m$ is needed and can be obtained from the following lemma.

\begin{lemma}\label{lemma:2.1}
Let $G$ be a finite group, $C(G)=\{H\leq G\mid H \text{ cyclic}\}$ and $p$ the smallest prime divisor of $|G|$. Then the following inequality holds:
\begin{equation}\label{eq:1}
h_m(G)\geq \frac{p|G|}{(p-1)|C(G)|+1}.
\end{equation}
Equality holds if and only if $G$ is a $p$-group.
\end{lemma}

\begin{proof}
Let $d_1=1, d_2=p,\dots, d_r$ be the orders of the elements in $G$.

For all $i\in\overline{1,r}$, we introduce the following notations:
$$n_i=|\{a\in G\mid o(a)=d_i\}| \mbox{ and } n_i'=|\{H\in C(G)\mid |H|=d_i\}|.$$

Then the following relations hold:
\[
\begin{aligned}
m(G)&=\sum_{i=1}^r \frac{n_i}{d_i}=\sum_{i=1}^r n_i'\frac{\varphi(d_i)}{d_i}=1+\sum_{i=2}^r n_i'\frac{\varphi(d_i)}{d_i}\\
&=1+\sum_{i=2}^r n_i'\prod_{q|d_i, q \text{ prime}}\left(1-\frac{1}{q}\right)\\&\leq 1+\frac{p-1}{p}(n_2'+\dots+n_r')=\frac{(p-1)|C(G)|+1}{p}\,,
\end{aligned}
\]
which give \eqref{eq:1}.

The equality case happens if and only if $p$ is the only prime divisor of $|G|$, which holds if and only if $G$ is a $p$-group.
\end{proof}

\begin{remark}
Since $|C(G)|\leq |G|$, \eqref{eq:1} gives the anticipated lower boundary for $h_m$ in terms of the smallest prime
divisor $p$ of $|G|$: \[h_m(G)\geq \frac{p|G|}{(p-1)|G|+1}\,.\]
\end{remark}

For finite $p$-groups $G$ we can establish when $h_m(G)\in\mathbb{N}$ through the following result.

\begin{theorem}\label{th:2.2}
Let $G$ be a finite $p$-group. Then $h_m(G)\in\mathbb{N}$ if and only if $G$ is cyclic of order $p^{\sum_{i=1}^s p^i}$, with $s\in\mathbb{N}^*$, or $G\simeq D_8$.
\end{theorem}

\begin{proof}
Let $|G|=p^n$, $\exp(G)=p^m$ and 
$$n_i=|\{a\in G\mid o(a)=p^i\}| \mbox{ and } n_i'=|\{H\in C(G)\mid |H|=p^i\}|, \,\forall i\in\overline{0,m}.$$

Lemma \ref{lemma:2.1} implies that
\[ h_m(G)=\frac{p^{n+1}}{(p-1)|C(G)|+1}\,,\]
which leads to the following equivalences:
\begin{equation}\label{hm(G)}
\begin{aligned} h_m(G)\in\mathbb{N}&\text{ if and only if } (p-1)|C(G)|+1=p^t \text{ with }t\leq n+1\\& \text{ if and only if } |C(G)|=p^{t-1}+p^{t-2}+\dots+p+1, \text{ with }t\leq n.\end{aligned}
\end{equation}
For $p=2$, this formula becomes 
\begin{equation}\label{p2}
|C(G)|=2^t-1.
\end{equation}
Going back to the proof, we distinguish the following two cases:

\begin{description}
\item[Case 1: $p$] odd. There are two possibilities: $G$ can be cyclic or not.
\
\begin{enumerate}[a)]
\item Suppose $G$ is not cyclic

Using Theorem 1.10 from \cite{book1}, it follows that $\begin{cases}n_1'\equiv p+1 \pmod{p^2}\\
n_2',\dots,n_m'\equiv 0\pmod p\end{cases}$ and so, recalling that the trivial subgroup is an element of $C(G)$, we obtain $|C(G)|\equiv 2 \pmod p$. Thus there are no solutions.

\item \label{case 1:cyclic}
Suppose $G$ is cyclic

From \eqref{hm(G)}, it follows that $n+1= p^{t-1}+p^{t-2}+\dots+p+1, \text{ with }t\leq n\xRightarrow{s=t-1} n=\sum_{i=1}^s p^i.$
\end{enumerate}
\item[Case 2: $p=2$]\label{case:2}
\
\begin{enumerate}[a)]

\item
Suppose $G$ is not cyclic and $G$ is not of maximal class

Using Theorem 1.17 from \cite{book1}, it follows that $\begin{cases}n_1'\equiv 3 \pmod{4}\\
n_2',\dots,n_m'\equiv 0\pmod 2\end{cases}$ and so $|C(G)|\equiv 0 \pmod 2$. Thus there are no solutions.

\item
Suppose $G$ is cyclic 

Therefore we have $|C(G)|=n+1$. Using \eqref{p2}, it follows that $n=2^t-2\xlongequal{s=t-1}\sum_{i=1}^s 2^i$ with $1\leq s\leq n+1$.

\item
Suppose $G$ is of maximal class 

It follows that $G\in\{D_{2^n}, Q_{2^n}, SD_{2^n}\}$. Using \cite{article3}, the following analysis is obtained:
\begin{enumerate}[i)]
\item
$G\cong D_{2^n}\Rightarrow |C(G)|=2^{n-1}+n\Rightarrow 2^{n-1}+n+1|2^{n+1}$. It follows that the left hand side has to be a power of $2$: $2^{n-1}+n+1=2^t$. This happens when $2^{n-1}=n+1$. Since $2^{n-1}> n+1, \forall n>3$ (it can be shown inductively), it can be proven by direct computation that the only solution is $n=3$ and therefore $G\cong D_8$.
\item
$G\cong Q_{2^n}\Rightarrow |C(G)|=2^{n-2}+n\Rightarrow 2^{n-2}+n+1|2^{n+1}$. Reasoning similarly as above, it follows that there are no solutions.
\item
$G \cong SD_{2^n}\Rightarrow |C(G)|=3\cdot 2^{n-3}+n\Rightarrow 3\cdot 2^{n-3}+n+1|2^{n+1}$. The left hand side must be a power of $2$: $3\cdot 2^{n-3}+n+1=2^s$, therefore $n+1$ is divisible by $2^{n-3}$. Using $2^{n-3}> n+1, \forall n>5$ (it can be again shown inductively), it can be proved that there are no solutions.
\end{enumerate} 
\end{enumerate}
\end{description}The proof of Theorem 2.3 is now complete. 
\end{proof}

We note that Theorem 2.3 gives the following characterization for $D_8$.

\begin{corollary}
$D_8$ is the only non-cyclic $p$-group with integer harmonic mean of element orders.
\end{corollary}

An alternative characterization is the following:

\begin{proposition}\label{D8}
$D_8$ is the only dihedral group with integer harmonic mean of element orders.
\end{proposition}

\begin{proof}
\begin{equation}\label{ass}
\text{Let }D_{2n}\text{ be a dihedral group of order }2n\text{ such that }h_m(D_{2n})\in\mathbb{N}^*.
\end{equation} 
Let $n=p_1^{n_1}\cdot p_2^{n_2}\cdots p_k^{n_k}$ be the decomposition of $n$ as a product of prime factors, where $p_1<p_2<\dots<p_k$. 

Let us write $D_{2n}=\{1,r,\dots,r^{n-1}, s, sr, \dots, sr^{n-1}\}$, where $r$ is the rotation of order $\frac{2\pi}{n}$ and $s$ is the reflection around a vertex and the center of a regular polygon with $n$-sides. Obviously, $\langle r\rangle \cong C_n$ and $o(sr^i)=2, \forall i\in\overline{0,n-1}$.

If we compute $h_m(D_{2n})$, we get \[h_m(D_{2n})=\frac{|G|}{\sum_{i=0}^{n-1} \frac{1}{o(r^i)}+\sum_{i=0}^{n-1}\frac{1}{o(sr^i)}}=\frac{2n}{m(C_n)+\frac{n}{2}}.\]

Let us denote \[ \alpha:=\frac{2n}{m(C_n)+\frac{n}{2}}\in\mathbb{N}^*\]
because of assumption \eqref{ass}.
Since $m(C_n)>0$, it follows that $\alpha<4$ and therefore three cases can occur:

\begin{enumerate}[a)]
\item
$\alpha=1$ 

Then $m(C_n)=\frac{3n}{2}$\,, contradicting $m(C_n)=\sum_{a\in C_n}\frac{1}{o(a)}\leq\sum_{a\in C_n}1=n$.

\item
$\alpha=2$

It follows that 
\begin{equation}\label{n2}
m(C_n)=\frac{n}{2}.
\end{equation}
Clearly $C_n\cong \prod_{i=1}^k C_{p_i^{n_i}}$. According to Proposition \ref{propm} e), it follows that
\begin{equation}\label{prod} m(C_n)=\prod_{i=1}^k m(C_{p_i^{n_i}}).
\end{equation} Each factor in the right-hand side can be computed. Fix $i\in\overline{1,k}$. For each $j\in\overline{1,n_i}$, there are $\varphi(p_i^j)$ elements of order $p_i^j$.
 Therefore $\displaystyle m(C_{p_i^{n_i}})=1+\sum_{j=1}^{n_i}\frac{\varphi(p_i^j)}{p_i^j}=1+\sum_{j=1}^{n_i}\frac{p_i^j(1-\frac{1}{p_i})}{p_i^j}=1+\sum_{j=1}^{n_i}\left(1-\frac{1}{p_i}\right)=\frac{p_i-1+1}{p_i}+\frac{n_i}{p_i}\left(p_i-1\right)=\frac{(n_i+1)(p_i-1)+1}{p_i}.$
Using \eqref{n2} and \eqref{prod}, it follows that \[ \prod_{i=1}^k \left((p_i-1)(n_i+1)+1\right)=\frac{1}{2} p_1^{n_1+1}\cdot \dots\cdot p_n^{n_k+1}.\]
Since the left-hand side is an integer, it follows that $p_1=2$, so
\[(n_1+2)\prod_{i=2}^k((p_i-1)(n_i+1)+1)=2^{n_1}p_2^{n_2+1}\dots p_k^{n_k+1}\]
Bernoulli's inequality gives:
\[p_i^{n_i+1}=\left(1+(p_i-1)\right)^{n_i+1}>(p_i-1)(n_i+1)+1, \forall i=\overline{2,k},\] therefore $n_1+2>2^{n_1}$, i.e. $n_1=1$.

If $k>1$ it follows that
\[ 3\prod_{i=2}^k \left((p_i-1)(n_i+1)+1\right)=2p_2^{n_2+1}\cdot\dots\cdot p_k^{n_k+1}, \] therefore $p_2=3$ and
\[ 3(2n_2+3)\prod_{i=3}^k \left((p_i-1)(n_i+1)+1\right)=2\cdot 3^{n_2+1}p_3^{n_3+1}\cdot \dots p_k^{n+1}.\]
Then $2n_2+3\geq 2\cdot 3^{n_2}$, which does not yield solutions. Thus the assumption that $k>1$ does not hold, and consequently the only solution is $k=1$ and $n_1=2$, which gives $n=4$.

\item
$\alpha=3$ 

The analysis is analogous to the previous case. Alternatively, the result follows from Theorem 2.6.
\end{enumerate}
\end{proof}

\begin{remark}
Let us note that we can build nilpotent groups $G$ with $h_m(G)\in\mathbb{N}$ as direct products of $p$-groups of the type of the ones in Theorem \ref{th:2.2}. 

In addition, we can construct non-nilpotent groups $G$ with this property, for example $G=\SL(2,3)\times  C_{7^7}$. The idea is to start with a non-nilpotent group $G_1$ and to do a direct product with a cyclic group $G_2$ such that the denominator of $h_m(G_1)$ reduces. In the example above:
\[h_m(G)=h_m(\SL(2,3))\cdot h_m(C_{7^7})=\frac{24}{7}\cdot 7^6=24\cdot 7^5.\]
\end{remark}

In what follows, we will study the integer values of the function $h_m$.

Obviously, we have $h_m(G)=1$ if and only if $G$ is the trivial group.

In addition, Proposition \ref{D8} and its proof give that $h_m(C_4)=h_m(D_8)=2$. We can show that these are the only groups $G$ with $h_m(G)=2$.

\begin{theorem}\label{th:2.5}
Let $G$ be a finite group. Then $h_m(G)=2$ if and only if $G\cong C_4$ or $G\cong D_8$.
\end{theorem}

\begin{proof}
Since $h_m(G)=2$, it follows that $2||G|$, therefore \eqref{eq:1} becomes 
\[ 2\geq \frac{2|G|}{|C(G)|+1}\,,\]
i.e. 
\begin{equation}\label{eq:2}
|C(G)|\geq |G|-1.
\end{equation}

Using the same notations as in Lemma \ref{lemma:2.1}, it follows that $d_2=2$ and 
\[|G|=\sum_{i=1}^r n_i=\sum_{i=1}^r n_i'\varphi(d_i)\]
\[|C(G)|=\sum_{i=1}^r n_i',\]
therefore \eqref{eq:2} becomes
\begin{equation}\label{9}
 \sum_{i=1}^r n_i'(\varphi(d_i)-1)\leq 1.\end{equation}

Since $\varphi(d_i)>1$, for $d_i>2$, we identify two possible cases:

\begin{enumerate}[a)]
\item
$r=2$ 

It follows that $G$ is an elementary abelian $2$-group. Then $h_m(G)\not\in\mathbb{N}$ by Theorem \ref{th:2.2}, which is a contradiction.

\item
$r=3$ 

Clearly, we have $n_3'=1$ and $\varphi(d_3)=2$, and so $d_3=3$ or $d_3=4$. 

If $d_3=3$, then $G$ has $1$ element of order $1$, $2$ elements of order $3$ and $|G|-3$ elements of order $2$. Thus $h_m(G)=\frac{6|G|}{3|G|+1}\notin\mathbb{N}$, which is a contradiction.

If $d_3=4$, then $h_m(G)=2$. This means that $G$ is a $2$-group of exponent $4$ with a single cyclic subgroup of order $4$. Using Theorem \ref{th:2.2}, we obtain  $G\cong C_4$ or $G\cong D_8$, as desired.
\end{enumerate}
\end{proof}

\begin{theorem}
The finite non-trivial groups $G$ with $h_m(G)\leq 2$ are $C_2^n, n\in\mathbb{N}, C_3, S_3, C_4$ and $D_8$. 
\end{theorem}
\begin{proof}
Using the same ideas and notations as in the proof of Theorem \ref{th:2.5}, we can classify the finite groups $G$ with $h_m(G)\leq 2$. We identify the following cases:
\begin{enumerate}[a)]
\item
$|G|$ is odd

It follows that 
\[\begin{aligned}
h_m(G) &=\frac{|G|}{m(G)}\geq \frac{|G|}{1+\frac{|G|-1}{p_1}}\geq \frac{|G|}{1+\frac{|G|-1}{3}}=\frac{3|G|}{|G|+2}\Rightarrow\\
\frac{3|G|}{|G|+2}\leq & 2\Rightarrow |G|\leq 4\Rightarrow |G|=3\Rightarrow G\cong C_3. &
\end{aligned}\]

\item
$|G|$ is even

It follows that \eqref{eq:2} holds. There are two possibilities:
\begin{enumerate}[i)]
\item
$r=2$ 

Then $G\cong C_2^n$, $n\in\mathbb{N}$.

\item
$r=3$ 

Then there are two possibilities:
\begin{description}
\item[$\mathbf{d_3=3}$]
Inequality \eqref{9} gives that $n_3'=1$ which means that $G$ has a unique cyclic subgroup of order $3$, let us denote this by $H$. Since $G$ does not contain cyclic subgroups of order $>3$, it follows that $H$ is the only $3$-Sylow subgroup of $G$. It is also normal. Since $|G|$ even, it follows that there is also a $2$-Sylow subgroup of order $2$. Thus $G=HK$. Since $G$ does not have cyclic subgroups of order $6$, it follows that $C_K(H)=1$, therefore $|K|=|{\rm Aut}(H)|=2$. We conclude that $G\cong S_3$.
\item[$\mathbf{d_3=4}$]
$G\cong C_4 \text{ or }G\cong D_8, \text{ for }d_3=4.$
\end{description}
\end{enumerate}
\end{enumerate}

This gives the conclusion. Moreover, we have $$\min\{ h_m(G)| G\text{ finite non-trivial}\}=\frac{4}{3}\,.$$
The minimum is obtained for $C_2$.
\end{proof}

Next we will focus on finite groups $G$ with $h_m(G)=3$. Note that the smallest example of such a group is $SmallGroup(12,1)$. We are not able to determine all these groups, but we can prove that they have even order and are not nilpotent. 

\begin{proposition}
There are no finite groups $G$ of odd order with $h_m(G)=3$.
\end{proposition}

\begin{proof}
Let $G$ be a finite group of odd order such that $h_m(G)=3$. Then $3||G|$. Let $H\leq G$ with $|H|=3$. We identify the following cases:
\begin{description}
\item[Case 1:] $\exp(G)=3$

Then $$h_m(G)=\frac{|G|}{1+\frac{|G|-1}{3}}=\frac{3|G|}{|G|+2}\neq 3,$$which is a contradiction.

\item[Case 2:] $\exp(G)\neq 3$

We will prove that
\begin{equation}\label{eq:3}
G\text{ has at least }6\text{ elements of order }\geq 5
\end{equation}
\end{description}

We identify the cases:
\begin{enumerate}[a)]
\item $3^2|\exp(G)$

Then $G$ has at least a cyclic subgroup of order $3^2$ and so at least $6=\varphi(3^2)$ elements of order $9$.

\item $3^2\nmid \exp(G)$.

Let $p$ the smallest prime $\neq 3$ such that $p|\exp(G)$. If $p\geq 7$, then $G$ has at least a cyclic subgroup of order $p$ and consequently at least $\varphi(p)=p-1$ elements of order $p$. If $p=5$, there are two sub-cases:
\begin{enumerate}[i)] 
\item $G$ has only one subgroup $K$ with $|K|=5$. It follows that $K\lhd G$, therefore $KH\leq G$ and $|KH|=15$. Then $KH$ is cyclic and it possesses $\varphi(15)=8$ elements of order $15$.
\item $G$ has at least $3$ subgroups of order $5$. Then $G$ has at least $3\varphi(5)=12$ elements of order $5$.
\end{enumerate}
This concludes the proof of \eqref{eq:3}. We get
\[h_m(G)\geq \frac{|G|}{1+\frac{6}{5}+\frac{|G|-7}{3}}=\frac{15|G|}{5|G|-2}>3,\] a contradiction which completes the proof. 
\end{enumerate}
\end{proof}

\begin{proposition}
There are no finite nilpotent groups $G$ with $h_m(G)=3$.
\end{proposition}

\begin{proof}
Assume that $G$ is a finite nilpotent group such that $h_m(G)=3$. If $G$ is a $p$-group, then $p=3$ and the conclusion follows from Proposition 2.9. If $G$ is not a $p$-group, then it can be written as a direct product of at least two $p$-groups, say $G=G_1\times\cdots\times G_k$ with $k\geq 2$. Since
\begin{equation}
h_m(G)=h_m(G_1)\cdots h_m(G_k),\nonumber
\end{equation}we get $h_m(G_i)<2$, $\forall\, i=1,...,k$. Now Theorem 2.8 implies that $G_i=C_2^n$ for some $n\in\mathbb{N}$ or $G_i=C_3$, and therefore $h_m(G_i)=\frac{2^{n+1}}{2^n+1}$ for some $n\in\mathbb{N}$ or $h_m(G_i)=\frac{9}{5}$\,. We remark that any product of these numbers is not $3$, contradicting our assumption.
\end{proof}

Finally, we note that the results so far leave the following open question:

\begin{question}
Which are the integer values contained in $\im(h_m)$? 
\end{question}
\bigskip

\begin{acknowledgements}
The authors are grateful to the reviewer for their remarks which improved the previous version of the paper.
\end{acknowledgements}

\begin{funding}
The authors did not receive support from any organization for the submitted work.
\end{funding}

\begin{conflicts}
The authors declare that they have no conflict of interest.
\end{conflicts}

\end{document}